\documentclass{aupl}

\usepackage{
amsfonts,
latexsym,
amssymb,
verbatim,
mathrsfs,
enumerate
}

\newcommand{\labbel}{\label}

\newtheorem*{theorem*}{Theorem}
\newtheorem*{proposition*}{Proposition}
\newtheorem*{corollary*}{Corollary}
\newtheorem*{lemma*}{Lemma}

\theoremstyle{definition}

\newtheorem*{definition*}{Definition}

\theoremstyle{remark}

\newcommand{\brfrt}{\hspace{0 pt}}

\def\v{\mathcal V}  % variety V

\newcommand{\alg}{\mathbf}

 \allowdisplaybreaks[1]

\begin{document}
 
\title{Nest-representable tolerances}

\author{Paolo Lipparini} 
\address{Nido di Matematica\\Viale della  Ricerca
 Scientifica\\Universit\`a di Roma ``Tor Vergata'' 
\\I-00133 ROME ITALY}
\urladdr{http://www.mat.uniroma2.it/\textasciitilde lipparin}


\thanks{Work performed under the auspices of G.N.S.A.G.A. Work 
partially supported by PRIN 2012 ``Logica, Modelli e Insiemi''}

\begin{abstract}
We introduce the notion of
a nest-representable tolerance and show that some results from \cite{contol}
and \cite{malgra} 
can be extended to this more general setting.
\end{abstract} 

\maketitle

\begin{definition*} \labbel{def}  
Recall from \cite{contol} that a tolerance
$\Theta$ of some algebra ${\alg A} $  is {\em representable} 
if and only if there exists a compatible and reflexive
relation $R$ on ${\alg A} $ such that 
$\Theta= R \circ R ^\smallsmile $ (where
 $R ^\smallsmile $ denotes the converse of $R$).
A tolerance
$\Theta$ of some algebra ${\alg A} $  is {\em weakly representable} 
if and only if there exists a set $K$ (possibly infinite) and there are
compatible and reflexive
relations $R_k$ ($k \in K$) on ${\alg A} $ such that 
$\Theta= \bigcap _{k \in K} (R_k \circ R_k ^\smallsmile )$. 

We define the set of \emph{nest-representable}
tolerances of ${\alg A} $ as the smallest set of tolerances of ${\alg A} $ 
which is closed under the following formation rules.
  \begin{enumerate}  
  \item  
Every representable tolerance is nest-representable.
\item
The intersection of any family of nest-representable
tolerances is nest-representable.
\item 
If $\Psi$ is a  nest-representable
tolerance and $R$ is a compatible and reflexive
relation, then  $R \circ \Psi \circ R ^\smallsmile  $
 is a nest-representable
tolerance. 
 \end{enumerate}
\end{definition*}

Notice that, in particular, every (weakly)
representable tolerance is nest-representable.
We shall show that many results from \cite{contol, malgra}
hold also for nest-representable tolerances, not only 
for   (weakly)
representable tolerances.

Recall that if  $p$ and $q$ are terms of the same arity for the 
language
$ \{\circ, \cap\} $, then a
strong Maltsev  condition $M(p \subseteq q)$ 
can be associated to the inclusion 
$p \subseteq q$, where the arguments of 
$p$ and  $q$ are intended to vary among congruences
of some algebra.  
See, e.~g., 
 Cz\'edli, Horv{\'a}th,  Lipparini \cite{CHL}, 
   Freese, McKenzie
  \cite[Chapter XIII]{FMK},
Hutchinson, Cz\'edli  \cite{HC},
J{\'o}nsson \cite{J},  Pixley 
\cite{P}, 
Wille \cite{W}. 
In this regard, we will follow the notations from Definition 6 
in \cite{contol}.
We can also consider $ \{\circ, \cap, +\} $-terms,
where $+$ is always interpreted as
$\Theta + \Psi = \bigcup_{n < \omega } \Theta \circ_n \Psi $.
Thus   $ \alpha + \beta $ is always interpreted as the join 
in the lattice of congruences, but $\Theta + \Psi $
turns out to be generally much larger than
the join of $\Theta$ and $ \Psi$ in \emph{the lattice of tolerances}.
If $p$ is a  $ \{\circ, \cap, +\} $-term,
we let $p_n$ denote the term obtained from 
$p$ by substituting $+$ with $\circ_{n}$.
By using $p_n$, we can express (not necessarily strong)
 Maltsev conditions. See \cite{CHL} for details; 
see the proof of \cite[Theorem 4]{contol}, as far as the notions
in the present note are concerned.     
  
Graphs provide a
 neat way to look at the Maltsev  condition
 associated to an inclusion. To each
$ \{\circ, \cap\} $-term it can be naturally associated a labeled 
graph. See Cz\'edli  \cite{czjsd, C2, czmsd, C4,C5} and Cz\'edli, Day
 \cite{CD}. In \cite{malgra} we observed that 
to every pair of edge-labeled graphs ${\mathbf G}$
and 
${\mathbf H}$
one can associate 
a condition
$M({\mathbf G},{\mathbf H})$,
in such a way that when
${\mathbf G}$
and 
${\mathbf H}$ are the graphs associated to the terms 
$p$ and $q$, then 
$M({\mathbf G},{\mathbf H})$
turns out to be equal to
 $M(p \subseteq q)$. 
In  \cite{contol} we introduced the notion of  a 
 \emph{regular}   $ \{\circ, \cap\} $-term.
Roughly, $p$ is regular if during the construction of $p$   
we never encounter two adjacent symbols. 
Correspondingly, an 
  edge-labeled graph ${\mathbf G}$ is \emph{regular}
if it is finite and, for each label,  all the equivalence classes of vertices 
which can be connected through edges with that label
have cardinality $\leq 2$. 
If  ${\mathbf G}$ is an edge-labeled graph 
with labels $\alpha_1, \dots, \alpha _n$ and 
with distinguished vertices $d_1, \dots, d _h$ and 
$R_1, \dots, R_n$
are symmetric and reflexive relations on some set $A$,
we let ${\mathbf G} (R_1, \dots, R_n)$ denote
the $h$-ary relation on $A$ consisting of those 
$h$-uples   $a_1, \dots, a _h$ of $A$ such that 
${\mathbf G}$ can be represented in $A$
in such a way that the distinguished vertices 
correspond to $a_1, \dots, a _h$,
and, for every label $i$, those  edges labeled by $\alpha_i$ are represented by 
$R_i$-related elements. See \cite[Definition 2]{malgra}.   
We refer to  \cite{contol}
and \cite{malgra} for more details and for
further unexplained notions and notations.

A special case of item (1) in the following theorem shall be presented 
in a planned expanded version of \cite{ntcm}.
Probably the following proof can be more easily understood 
through that example. 

\begin{theorem*}\label{contolnest} 
(1) Suppose that $\v$ is a variety and that 
 $p$ and $q$ are terms of the same arity.
Suppose that either (i) $p$ and  $q$ are
$ \{\circ, \cap\} $-terms and $p$ is regular,
or (ii) $p$ and  $q$ are
$ \{\circ, \cap, +\} $-terms and
either $p_3$ or   $p_4$ is regular. 
 Then 
the following conditions are equivalent.
  \begin{enumerate}[(a)]
\item 
$\v $ satisfies the congruence identity
$p(\alpha_1, \dots, \alpha_n) \subseteq q(\alpha_1, \dots, \alpha_n )$.
\item
 $\v$ satisfies the (strong in case (i)) Maltsev  condition $M(p \subseteq q)$.
\item
 The tolerance identity
$p(\Theta_1, \dots, \Theta_n) \subseteq q(\Theta_1, \dots, \Theta_n)$
holds for every algebra $ {\alg A} $ in 
$\v$ and
for all nest-representable tolerances
$\Theta_1, \dots, \Theta_n $ of $ {\alg A} $.
   \end{enumerate} 

(2) Suppose that $\v$ 
is a variety and  ${\mathbf G}$, ${\mathbf H}$ are labeled graphs
with the same labels and with the same number of distinguished vertices.
If ${\mathbf G}$ 
is regular, then the following are equivalent.
  \begin{enumerate}[(a)]
\item
 $\v$ satisfies ${\mathbf G} (\alpha _1, \dots, \alpha _n)\subseteq {\mathbf H}(\alpha _1, \dots, \alpha _n)$ for congruences.
\item
$\v$ satisfies the condition $M({\mathbf G},{\mathbf H})$.
\item
 $\v$ satisfies ${\mathbf G} (\Theta _1, \dots, \Theta _n)\subseteq {\mathbf H}(\Theta _1, \dots, \Theta _n)$ for nest-representable tolerances.
 \end{enumerate} 
\end{theorem*}

\begin{proof}
We shall prove (1)(i); the case (1)(ii) 
then follows by arguments similar to 
\cite[proof of Theorem 4]{contol}.
The proof of (2)
is entirely similar. 
The implication (a) $\Rightarrow $  (b) in (1) is classical, and
(c) $\Rightarrow $ (a) is  trivial.  

In order to prove 
(b) $\Rightarrow $  (c), we shall assume
 the notations 
from the proof of 
\cite[Theorem 3 (ii) $\Rightarrow $  (iii)]{contol}.
We have to show  that
$d_w =
t_w(c_1, \dots, c_m) \mathrel {\Theta_i}  
t _{w'}(c_1, \dots, c_m) =
d_{w'}$,
whenever the vertices
$w, w' \in W$ of ${\mathbf G}_q$ are connected by an edge labeled by
$ \alpha_i$,
using  the same 
assumptions of 
\cite[Theorem 3]{contol},
 except  that $\Theta_i$ is 
only assumed to be  nest-representable.   
Fix some $i$ and say that two indices $j,h \leq m$ 
are \emph{paired} if  $ \{ v_j, v_h \} $ is a $\sim_i$-equivalence class.
In particular,  we have that if 
$j $ and  $h$ are paired, 
then $ c_j \mathrel \Theta_i  c_h $.

Suppose that $\Psi$ is a nest-representable
tolerance of ${\alg A} $.
We are going to prove, by induction on the 
complexity of the nest-\brfrt representation of $\Psi$,
that if  $e_1, \dots, e_m$ 
are elements such that 
 $ e_j \mathrel \Psi e_h $,
whenever 
$j $, $h$ are paired,
then 
$t_w(e_1, \dots, e_m) \mathrel {\Psi}  
t _{w'}(e_1, \dots, e_m)$. 
The special case $\Psi = \Theta_i$, 
$c_1 = e_1$, \dots,  $c_m = e_m$
will then give the desired result.

The ``basis'' case given by (1) in the above definition,
that is, the case when $\Psi$ is representable
is given by the proof of  \cite[Theorem 3]{contol}.  
The proof is similar to the argument in  case (3) below, 
considering $\Phi=0$ there, that is, being $\Phi$-related means to be equal. 
Then apply identity (m$ _{w, w',i} $).

If the nest-representability of $\Psi$
is given by case (2) above, that is, 
$\Psi= \bigcap _{k \in K_i} \Phi _{k} $
and, by the inductive hypothesis, we have 
$t_w(e_1, \dots, e_m) \mathrel { \Phi _k}  
t _{w'}(e_1, \allowbreak \dots, e_m)$, for every $k \in  K_i$,
 then obviously
$t_w(e_1, \dots, e_m) \mathrel {\Psi}  
t _{w'}(e_1, \dots, e_m)$. Thus the induction step
is complete in this case. Cf. also the proof of
\cite[Theorem 3 (ii) $\Rightarrow $  (iii)$'$]{contol}.

Finally, 
if the nest-representability of $\Psi$
is given by case (3), then
  $\Psi = R \circ \Phi \circ R ^\smallsmile  $, for
some
 compatible and reflexive $R$ and 
some nest-representable $\Phi$
for which the induction has already been carried over.
If 
$j  \neq h$ are paired,
then  $ e_j \mathrel \Psi e_h $ and, 
by the above assumption on  $ \Psi$,
there are elements 
$b _{ijh} $ and $b' _{ijh} $ such that
$ e_j  \mathrel { R} b _{ijh} \mathrel  \Phi b' _{ijh}
\mathrel {  R ^\smallsmile} e_h$, thus
$e_h \mathrel {  R } b' _{ijh}$. 
Define elements $e^*_1, \dots, e^*_m$ as follows.
If  $ \{ v_j \} $ is a $\sim_i$-equivalence class,
let $e^*_j = e_j$.
If   $j< h$ are paired, let  
$e^*_j = b _{ijh}$ and
$e^*_h = b' _{ijh}$, thus $e^*_j \mathrel { \Phi} e^*_h $.
The above cases do not overlap, since 
$\sim_i$ is an equivalence relation; moreover, they 
cover all indices, by the assumption that $p$
is a regular term; see \cite{contol}. 
By the inductive hypothesis, 
$t_w(e^*_1, \dots, e^*_m) \mathrel {\Phi}  
t _{w'}(e^*_1, \dots, e^*_m)$.
Since $e_j \mathrel { R } e^*_j$, for every $j$, then  
$t_w(e_1, \dots, e_m) \mathrel {R}  
t_w(e^*_1, \dots, e^*_m) \mathrel {\Phi}  
t _{w'}(e^*_1, \dots, e^*_m) \mathrel { R ^\smallsmile } 
t _{w'}(e_1, \dots, e_m)$, hence
$t_w(e_1, \dots, e_m) \mathrel {R \circ \Phi \circ  R ^\smallsmile } 
t _{w'}(e_1, \dots, e_m)$, 
that is, 
$t_w(e_1, \dots, e_m) \mathrel { \Psi } 
t _{w'}(e_1, \dots, e_m)$, what we had to show.
 \end{proof}

Notice that (1)(i) can be seen as a particular case of
(2), by the  mentioned way of associating a graph to a term.

This is a preliminary version, it might contain inaccuraccies (to be
precise, it is more likely to contain inaccuracies than  subsequent versions).

\end{document}